\documentclass[a4paper,10pt, leqno]{article}
\usepackage[top=2cm,bottom=2cm]{geometry}
\usepackage{latexsym}
\usepackage{amsmath}
\usepackage{amsthm}
\usepackage{amssymb}
\usepackage{mathrsfs}
\usepackage{stmaryrd }
\usepackage{bm}
\usepackage{graphicx}
\usepackage{wrapfig}
\usepackage{float}
\usepackage{fancybox}
\usepackage{enumerate}
\usepackage{cite}
\usepackage[pagebackref,citecolor=blue,linkcolor=red,colorlinks=true]{hyperref}
\usepackage{calligra}
\usepackage{esint}

\newtheorem{Theorem}{Theorem}[section]

\newtheorem{Corollary}[Theorem]{Corollary}
\newtheorem{Proposition}[Theorem]{Proposition}
\newtheorem{Lemma}[Theorem]{Lemma}

\newtheorem{Res*}{}

\theoremstyle{definition}
\newtheorem{Definition}{Definition}[section]
\newtheorem{Remark}{Remark}

\DeclareMathAlphabet{\mathcalligra}{T1}{calligra}{a}{c}
\newcommand{\Cc}{\kern-4pt\mathcalligra{C\kern0.2pt a\kern0.2pt c\kern0.2pt c}\kern1pt}
\numberwithin{equation}{section}
\newcommand\ttimes{\cone}

\newcommand\res{\mathop{\hbox{\vrule height 7pt width .3pt depth 0pt
\vrule height .3pt width 5pt depth 0pt}}\nolimits}

\newcommand{\cH}{{\mathcal{H}}}

\newcommand{\bB}{{\mathbf{B}}}

\newcommand\Z{{\mathbb Z}}

\newcommand{\eps}{{\varepsilon}}

\def\XXint#1#2#3{{\setbox0=\hbox{$#1{#2#3}{\int}$ }
\vcenter{\hbox{$#2#3$ }}\kern-.6\wd0}}

\newcommand{\cone}{{\times\hspace{-0.22cm}\times}}

\newcommand{\dist}{{\rm {dist}}}

\def\a#1{\left\llbracket{#1}\right\rrbracket}

\title{Uniqueness of boundary tangent cones for $2$-dimensional area-minimizing currents}
\author{Camillo De Lellis, Stefano Nardulli, Simone Steinbr\"uchel}

\date{}
\begin{document}

\maketitle

\begin{abstract}
In this paper we show that, if $T$ is an area-minimizing $2$-dimensional integral current with $\partial T = Q \a{\Gamma}$, where $\Gamma$ is a $C^{1,\alpha}$  curve for $\alpha>0$ and $Q$ an arbitrary integer, then $T$ has a unique tangent cone at every boundary point, with a polynomial convergence rate. The proof is a simple reduction to the case $Q=1$, studied by Hirsch and Marini in \cite{HM}.
\end{abstract}

\section{Introduction}\label{s:1}

The main goal of this note is to prove the following theorem, we refer to \cite{Federer,Simon} for the relevant notation and definitions.

\begin{Theorem}\label{t:main}
Let $T$ be a $2$-dimensional area-minimizing integral current $T$ in some open set $U\subset \mathbb R^2$ and assume that $\partial T = Q \a{\Gamma}$ for some integer $Q$ and some $C^{1,\alpha}$ embedded simple curve with $\alpha >0$. Then at every point $x\in \Gamma$ there is a unique tangent cone $C$, that is, if $\iota_{x, r}: \mathbb{R}^{m+n} \rightarrow \mathbb{R}^{m+n}$ is the map $z \mapsto \frac{z-x}{r}$, then 
$T_{x,r}:= (\iota_{x,r})_\sharp T$ converges, as $r\downarrow 0$, to $C$. 
\end{Theorem}

A corresponding interior theorem was proved by White almost 40 years ago in his groundbreaking paper \cite{White}. White's approach has been extended to other cases in the interior, see e.g. \cite{Chang,PuRi}, while a comprehensive generalization to a rather robust notion of almost minimizer (which covers all known situations) has been given in \cite{DSS}. 

Concerning the boundary case, for $Q=1$ Theorem \ref{t:main} has been shown recently by Hirsch and Marini in \cite{HM}, still building on White's seminal approach. Our proof deviates only slightly from the one of \cite{HM}, however our interest in Theorem \ref{t:main} comes from the more general problem of proving regularity at the boundary for area-minimizing currents in higher codimension, when the multiplicity $Q$ is higher than $1$. The latter problem was raised by White in \cite{Collection} and there are no available results thus far, excluding the trivial case of $1$-dimensional currents. Theorem \ref{t:main} is a stepping stone for a subsequent work of us, \cite{DNS}, where we give a first result on White's problem: in \cite{DNS} we prove full regularity at the boundary for $2$-dimensional area-minimizing currents, under the assumption that $\Gamma$ is sufficiently smooth and is contained in the boundary of a sufficiently smooth uniformly convex set (more specifically $C^{3,\alpha}$ regularity of both suffices). Thus \cite{DNS} partially generalizes a theorem of Allard, who in his seminal paper \cite{AllB} proved full regularity under the above convexity assumption when $Q=1$, in any dimension and codimension. We will dedicate the next section to give a more general and more precise statement than the one above, where the assumption is relaxed to a suitable form of almost minimality and the convergence rate of $T_{x,r}$ to $C$ is shown to be polynomial. While the first point is less relevant for the purpose of \cite{DNS}, the second plays indeed a fundamental role. 

\subsection{Acknowledgements}
The first author acknowledges the support from the National Science Foundation through  the  grant  FRG-1854147.
The second author would like to thank Fapesp for financial support via the grant ``Bolsa de Pesquisa no Exterior'' number 2018/22938-4. 

\section{Almost minimality and polynomial rate of convergence}

Before stating our main theorem we establish some notation. First of all we refer the reader to \cite{Federer,Simon} for the notation and basic terminology in the theory of integral currents. We will use the short hand notation $T_{x, r}$ for $\left(\iota_{x, r}\right)_{\sharp} T$ (and drop $x$ if it is the origin) and, for a given $C^1$ non selfintersecting curve, we will denote by $T_x\Gamma$ its tangent line at $x$.

The exact definition of almost minimality which we will be used in the rest of the notes is as follows. 

\begin{Definition}[{\cite[Definition 1.1]{HM}}]\label{d:1.1} 
Given three real numbers $\Lambda \ge0, r_{0}\in ]0,1], \alpha_{0}>0$, we say that an $m$-dimensional integral current $T$ 
is $(\Lambda, r_{0}, \alpha_{0})$-\textbf{\emph{almost $($area$)$ minimizing at $x \in \operatorname{spt}(T)$}}, if we have 
\begin{align}\label{e:1.1}
 \|T\|\left(\bB_{r}(x)\right) \leq\left(1+\Lambda r^{\alpha_{0}}\right)\|T+\partial\tilde{T}\|\left(\bB_{r}(x)\right),
\end{align}
for all $0<r<r_{0}$ and all integral $(m+1)$-dimensional currents $\tilde{T}$ supported in $\bB_{r}(x)$. 
\end{Definition}

The convergence of integral currents will be measured using the flat distance between $T, S \in \mathbf{I}_{m}\left(\bB_{R+1}\right)$, (here we use the definition of \cite[Section 6.7]{Simon} which is different from Federer's definition, cf. \cite{Federer}):
\begin{align}\label{e:1.2}
  \mathbf{d}_{\bB_{R}}(T, S)=\inf \left\{\|R\|\left(\bB_{R}\right)+\|Q\|\left(\bB_{R}\right): T-S=R+\partial Q \text { in } \bB_{R+1}, \right\},
\end{align}
where the infimum is taken over $R \in \mathbf{I}_{m}\left(\bB_{R+1}\right)$ and $Q \in \mathbf{I}_{m+1}\left(\bB_{R+1}\right)$.
We also use ${\rm dist}_H$ to denote the Hausdorff distance between closed sets and we denote by $e (p, r)$ the usual ``spherical excess'' of a current $T$, namely
\begin{equation}\label{e:excess}
e (p,r) := \frac{\|T\| (\bB_r (p))}{\pi r^2} - \Theta (T, p)\, ,
\end{equation}
where 
\[
\Theta (T,p) := \lim_{r\downarrow 0} \frac{\|T\| (\bB_r (p))}{\pi r^2}
\]
(the latter limit will be shown to exist in the next section). Finally, we measure the H\"older regularity of the curve $\Gamma$ with a standard H\"older seminorm, 
\[
[\Gamma]_{0,\alpha, U} := \sup_{q\neq p\in \Gamma \cap U} \frac{|T_p \Gamma - T_q \Gamma|}{|p-q|^\alpha}\, .
\]

\begin{Theorem}\label{t:main-preciso}
There are constants $C$, $\varepsilon_0$, and $\beta>0$ depending only on $n, \alpha$, and $\Theta_0$ with the following property.  
Assume that:
\begin{itemize}
\item[(a)] $\Gamma$ is a $C^{1, \alpha}$ non self-intersecting curve in $\bB_r (p)$ with $p\in \Gamma$, $r\leq 1$, and $r^\alpha [\Gamma]_{0, \alpha, \bB_r (p)} \leq \varepsilon_0$;
\item[(b)] $T$ is a two-dimensional integral current in $\bB_r (p) \subset \mathbb{R}^{2+n}$ with $\partial T=Q\a{\Gamma}$;
\item[(c)]  $T$ is $(\Lambda, r, \alpha)$-almost minimizing at $p$ and $\Lambda r^\alpha \leq \varepsilon_0$;
\item[(d)] $\Theta (T, p) = \Theta_0$ and $e(p,r) \leq \varepsilon_0^2$.
\end{itemize}
Then there exists a unique area minimizing cone $S$ such that $\partial S = Q \a{T_0 \Gamma}$ and moreover for every $0<\rho \leq r$ we have
\begin{align}
|e (p, \rho)| &\leq C |e (p,r)| \left(\frac{\rho}{r}\right)^{2\beta} + C (\Lambda^2 + [\Gamma]^2_{0, \alpha, \bB_r (p)}) \left(\frac{\rho}{r}\right)^{2\beta}\label{e:decay-excess}\\
\mathbf{d}_{\bB_1 (0)} (T_{p,\rho},S) &\leq C |e(p,r)|^{\frac{1}{2}}  \left(\frac{\rho}{r}\right)^{\beta} + C (\Lambda + [\Gamma]_{0, \alpha, \bB_r (p)}) \left(\frac{\rho}{r}\right)^{\beta}\\
\dist_H ({\rm spt}\, (T_{p, \rho}) \cap \bar\bB_1, {\rm spt}\, (S) \cap \bB_1) &\leq C |e(p,r)|^{\frac{1}{2}}  \left(\frac{\rho}{r}\right)^{\beta} + C (\Lambda + [\Gamma]_{0, \alpha, \bB_r (p)}) \left(\frac{\rho}{r}\right)^{\beta}
\end{align}
\end{Theorem}

Theorem \ref{t:main} follows easily from the above more precise version and the classical monotonicity formula (which we recall below). Moreover, a simple scaling argument reduces it to the case $p=0$ and $r=1$, which will be the focus of the rest of this note.

\section{Straightening the boundary and monotonicity formula}

Following \cite{HM}, a standard computation which ``straightens'' the boundary with a suitable $C^{1,\alpha}$ diffeomorphism reduces Theorem \ref{t:main-preciso} to the case in which $\Gamma$ is a straight line. We report here the relevant technical lemma, which corresponds to \cite[Lemma 2.1]{HM}, and clearly holds in any dimension.

\begin{Lemma}\label{Lemma:2.1}
For every $\alpha_1\in (0, 1)$ and $m,n$ integers there exist constants $\varepsilon_1 = \varepsilon_1(m,n,\alpha_1)>0$ and $C_1 (m,n,\alpha_1)$, with the following properties.
Whenever $\Gamma$ is the graph of a $C^{1, \alpha_1}$ function $\psi: \mathbb R^m\supset B_1 \to \mathbb R^n$ with $\psi(0)=0$, $D\psi (0)=0$ and $[D\psi]_{\alpha_1, B_1} \leq \varepsilon_1$, there is a $C^{1,\alpha_1}$ diffeomorphism $\phi: \bB_1 \to \bB_1$ such that
\begin{itemize}
\item[(i)] $\phi$ maps each $\partial \bB_\rho$ onto itself;
\item[(ii)] $\phi$ maps $\Gamma \cap \bB_1$ onto $\bB_1 \cap (\mathbb R^m \times \{0\})$;
\item[(iii)] $[D\phi]_{\alpha_1, \bB_1} \leq C_1 [D\psi]_{\alpha_1, B_1}$;
\item[(iv)] $|x|^{-1} |\phi(x)-x|+|D \phi(x)-{\rm Id}| \leq C_1 [D\psi]_{\alpha_1, B_1} |x|^{\alpha_1}$ for all $x\in \bB_1$.
\end{itemize}
\end{Lemma}

Observe that, if $\phi$ is as in the above lemma and $T$ is an integer rectifiable current, then
\[
(1 - C [\psi]_{\alpha_1, B_1} \rho^{\alpha_1}) \|T\| (\bB_\rho) \leq \|\phi_\sharp T\| (\bB_\rho) \leq (1+ C [\psi]_{\alpha_1, B_1} \rho^{\alpha_1}) \|T\| (\bB_\rho)\, .
\]
In particular, after applying a translation, a rescaling and a composition with a diffeomorphism $\phi$ as in the above lemma, it is immediate to check that Theorem \ref{t:main-preciso} can be reduced to the case in which $p=0$, $r=1$ and $\Gamma$ is a straight line $\ell$. More precisely, we only have to prove the following case of Theorem \ref{t:main-preciso}

\begin{Theorem}\label{t:main-preciso-2}
The conclusions of Theorem \ref{t:main-preciso} apply to the special case in which $p=0$, $r=1$, and $\Gamma$ is a straight line $\ell$ passing through the origin.
\end{Theorem}

We also record another application of Lemma \ref{Lemma:2.1}, given in \cite[Proposition 4.3]{HM}: the argument given in \cite{HM} under the assumption that $\partial T = \a{\Gamma}$ (i.e. $Q=1$) works indeed in our case as well.

\begin{Proposition}\label{t:2.3}
Let $\Gamma$ be a straight line passing through the origin, $p=0$, $r=1$ and $T$ be as in Theorem \ref{t:main-preciso}. Then there is a constant $C_2 (\alpha, n)$ such that
\begin{align}\label{e:2.3}
 e^{C_2 \Lambda \sigma^\alpha} \frac{\|T\|\left(\bB_{\sigma}\right)}{\sigma^2}-e^{C_2 \Lambda s^{\alpha}} \frac{\|T\|\left(\bB_{s}\right)}{s^{2}} \geq \int_{\bB_{\sigma} \backslash \bB_{s}} e^{C_{2}\Lambda |z|^{\alpha}} \frac{\left|z^{\perp}\right|^{2}}{2|z|^{4}} d\|T\|(z), 
\end{align}
for every $0<s< \sigma \leq 1$.
\end{Proposition}

In fact it can be readily checked that the argument given in \cite{HM} does not depend on the dimension of the current.

\begin{Remark}\label{r:2.1}
 A classical consequence of the almost monotonicity formula is that, under the assumptions of Theorem \ref{t:main-preciso} the density of the current $T$ exists in $p$, i.e. the limit
\begin{align}\label{e:2.5}
 \Theta (T, p)=\lim _{r \rightarrow 0} \frac{\|T\|\left(\bB_{r}(p)\right)}{\pi r^{2}}.
 \end{align}
 Moreover, the function $\Theta (T, \cdot)$ is upper-semicontinuous on $\Gamma$. The proof follows the classical arguments for area-minimizing currents and can be found in \cite{HM}.
 
Another consequence is that $\frac{\|T\| (\bB_\rho (p))}{\pi \rho^2} \geq e^{-C \Lambda \rho^\alpha} \Theta (T, p) \geq \Theta (T, p) - C \Lambda \rho^\alpha \Theta (T,p)$, so that $e (p,r)\geq - C \Lambda \rho^\alpha \Theta (T,p)$. In particular \eqref{e:decay-excess} is equivalent to the same statement with $e (p, \cdot)$ replacing $|e (p, \cdot)|$.  
\end{Remark}

\section{Classification of $2$-dimensional area-minimizing cones with a straight line boundary}

In this section we give a complete description of $2$-dimensional area minimizing cones with boundary $Q\a{\ell}$ for some straight line $\ell$. 

\begin{Proposition}\label{p:classification}
Let $S$ be a $2$-dimensional integral current in $\mathbb R^{2+n}$ such that
\begin{itemize}
\item[(a)] $S$ is area minimizing;
\item[(b)] $S$ is a cone, i.e. $(\iota_{0,r})_\sharp T = T$ for every $r>0$;
\item[(c)] $\partial S = Q \a{\ell}$ for some positive integer $Q$ and a straight line $\ell$ containing the origin.
\end{itemize}
Then we can decompose $S = S^{\rm int} + S^b$ where:
\begin{itemize}
\item[(i)] $S^{\rm int}$ and $S^b$ are both area minimizing and their supports intersect only at the origin;
\item[(ii)] $\partial S^{\rm int} = 0$ and thus 
\[
S^{\rm int} = \sum_{i=1}^N Q_i \a{V_i}
\]
where $Q_1, \ldots, Q_N$ are positive integers and $V_1, \ldots, V_N$ are distinct oriented $2$-dimensional planes such that $V_i \cap V_j = \{0\}$ for all $i\neq j$;
\item[(iii)] One of the following two alternatives hold for $S^b$:
\begin{itemize}
\item[(iiia)] There are $M$ distinct oriented halfplanes $V^+_1, \ldots, V^+_M$ and $M$ positive integers $Q_i^+$ such that $\partial \a{V^+_i} = \a{\ell}$ and $S^b = \sum_i Q_i^+ \a{V_i^+}$;
\item[(iiib)] There is a single oriented plane $V$ containing $\ell$, subdivided by the latter in two oriented half-planes $V^+$ and $V^-$ with $\partial \a{V^+} = - \partial \a{V^-} = \a{\ell}$, and two positive integers $Q^+$ and $Q^-$ such that $S^b = Q^+ \a{V^+} + Q^- \a{V^-}$.
\end{itemize}
\end{itemize}
\end{Proposition}

\begin{figure}[!htb]
    \centering
    \begin{minipage}{.5\textwidth}
        \centering
        \includegraphics[width=0.9\linewidth]{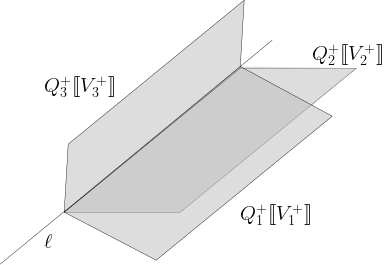}
        \caption{An illustration of $S^b$ in $(iiia)$\\ with $M=3$.}
        \label{fig:open_book}
    \end{minipage}%
    \begin{minipage}{0.5\textwidth}
        \centering
        \includegraphics[width=0.9\linewidth]{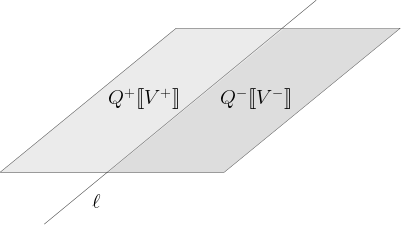}
        \caption{An illustration of $S^b$ in $(iiib)$.}
        \label{fig:closed_book}
    \end{minipage}
\end{figure}

Observe that in the first case we must have $\sum_i Q_i^+ = Q$ and thus $M\leq Q$, while in the second case we must have $Q^+-Q^-=Q$ and thus $Q^+ > Q^-$. Before coming to the proof of Proposition \ref{p:classification} we come to a simple corollary.

\begin{Corollary}\label{c:compactness}
The density at the origin of a $2$-dimensional area-minimizing cone as in Proposition \ref{p:classification} is always a positive multiple of $\frac{1}{2}$ and it is at least $\frac{Q}{2}$. If for every $Q$, $\bar{Q}$, and $n$ positive integers we denote by $\mathscr{C} (Q, \bar Q, n)$ the space of $2$-dimensional area-minimizing cones $S$ in $\mathbb R^{2+n}$ such that
\begin{itemize}
\item $\partial S = Q \a{\ell}$ for some line $\ell$ passing through the origin,
\item $\Theta (S, 0) \leq \frac{\bar Q}{2}$,
\end{itemize}
then $\mathscr{C} (Q, \bar Q, n)$ is compact in the topology of local flat convergence.
\end{Corollary}

\begin{proof}[Proof of Proposition \ref{p:classification}]
Recall that $R := \partial (S\res \bB_1) - Q \a{\ell} \res \bB_1$ is a $1$-dimensional integral current with $\partial R = Q \a{P_1} - Q \a{P_2}$, where $P_1$ and $P_2$ denote the two intersection points of $\ell$ with the unit sphere $\partial \bB_1$. Observe also that, by the interior regularity theory, $R$ is in fact a smooth $1$-dimensional embedded minimal submanifold of $\partial B_1\setminus \{P_1,P_2\}$. The support of $R$ consists therefore of the union of a finite number of distinct great circles $\gamma_1, \ldots, \gamma_N$ contained in $\partial B_1\setminus \{P_1,P_2\}$ and a finite number of distinct half great circles $\eta_1, \ldots, \eta_M$, each with endpoints $P_1$ and $P_2$. We denote by $V_1, \ldots, V_N$ the $2$-dimensional planes containing $\gamma_1, \ldots, \gamma_N$, for which we choose an arbitrary orientation, and by $V_1^+, \ldots, V_M^+$ the $2$-dimensional halfplanes containing $\eta_1, \ldots, \eta_M$, for which we fix the orientation such that $\partial \a{V_i^+} = \a{\ell}$. By the constancy theorem we must have
\[
S := \underbrace{\sum_i Q_i \a{V_i}}_{=:S^{\rm int}} + \underbrace{\sum_j \bar Q_j \a{V_j^+}}_{=:S^b}\, ,
\] 
for some nonzero integers $Q_1, \ldots, Q_N, \bar{Q}_1, \ldots , \bar{Q}_M$.
Without loss of generality we can change the orientation of the $V_i$ so that all the $Q_i$ are positive. $S^{\rm int}$ satisfies therefore the claims of the proposition. As for $S^b$ we distinguish two cases:
\begin{itemize}
\item[(a)] The $\bar{Q}_j$ are all positive. $S^b$ satisfies, therefore, the description of case $(iiia)$.
\item[(b)] One $\bar{Q}_j$, say $\bar{Q}_1$, is negative, and we denote it by $-\bar{Q}$ with $\bar Q$ positive. Observe that, since $\partial S^b = \sum_i \bar{Q}_i \a{\ell} = Q \a{\ell}$, we must have $\sum_i \bar{Q}_i = Q$ and thus at least one $\bar{Q}_j$ is positive. Fix any of them. We argue then as in \cite[Lemma 3.18]{DDHM}. We observe first that we can decompose $S^b$ as 
\[
S^b = \underbrace{-\a{V_1^+} + \a{V_j^+}}_{=:R_1} + \underbrace{(\bar Q_j -1) \a{V_j^+} - (\bar{Q}-1) \a{V_1^+} + \sum_{i\not\in \{1, j\}} \bar{Q}_i \a{V_i^+}}_{=:R_2}\, .
\]
It is clear that $\|S^b\|= \|R_1\|+ \|R_2\|$ and that $\partial R_1 = 0$. It thus follows that $R_1$ is an area-minimizing cone without boundary. It is therefore regular except possibly at the origin, which implies that $V_1^+$ and $V_j^+$ are two halves of the same plane. Since this argument can be applied to every pair $l,k$ with $\bar Q_l < 0 < \bar Q_k$, it follows then that indeed $M=2$, $\bar{Q}_1<0<\bar{Q}_2$. If we then let $V$ be the two-dimensional plane of which $V_1^+$ and $V_2^+$ are the two halves, claim $(iiib)$ readily follows after we let $V^+$ be $V_2^+$ and $V^-$ be the opposite (in terms of orientation) of $V_1^+$.\qedhere
\end{itemize}
\end{proof}

\section{White's and Hirsch-Marini's epiperimetric inequalities}

The proof of Theorem \ref{t:main-preciso} is based on a so-called epiperimetric inequality. 

\begin{Definition}\label{Def:EpiperimetricInequality} Let $S$ be a $2$-dimensional area-minimizing cone in $\mathbb R^{2+n}$ with $\partial S=Q\a{\ell}$ for some straight line $\ell$ passing through the origin and some positive integer $Q$, and let $R := \partial (S\res \bB_1) - Q \a{\ell} \res \bB_1$ be its cross section. Given two positive constants $\varepsilon$ and $\delta$ we say that $S$ satisfies the \textbf{\emph{$(\delta,\varepsilon)$-epiperimetric property}}, if the following holds. Let $Z \in \mathbf{I}_{1}\left(\partial \bB_{1}\right)$ be such that
 \begin{enumerate}[(i)]
 \item $\partial Z=\partial R$,
 \item $\mathcal{F}(Z-R) \leq \eps$,
\item $\|Z\| (\partial \bB_1) - \|R\| (\partial \bB_1) \leq \eps$,
\item $\operatorname{dist}(\operatorname{spt}(Z), \operatorname{spt}(R)) \leq \eps$.
\end{enumerate}
Then there is $H \in \mathbf{I}_{2}\left(\bB_{1}\right)$ such that $\partial H=Z-Q \a{\ell}\res \bB_1$ and
\begin{align}\label{e:3.1}
\|H\|\left(\bB_{1}\right)-\|S\|\left(\bB_{1}\right) \leq\left(1-\delta\right)\left(\| \a{0} \ttimes Z\|\left(\bB_{1}\right)-\|S\|\left(\bB_{1}\right)\right).
\end{align} 
\end{Definition}

\begin{Remark} Observe that as it is immediate to see $\a{0} \ttimes Z= \a{0} \ttimes \partial H$, and so \eqref{e:3.1} is the same as the following equation 
\begin{equation}
\|H\|\left(\bB_{1}\right)-\|S\|\left(\bB_{1}\right) \leq\left(1-\delta\right)\left(\|\a{0} \ttimes\partial H\|\left(\bB_{1}\right)-\|S\|\left(\bB_{1}\right)\right).
\end{equation}
\end{Remark}

The main point in the proof of Theorem \ref{t:main-preciso-2} is then that {\em every} $2$-dimensional area-minimizing cone satisfies an epiperimetric inequality.

\begin{Proposition}\label{p:epiperimetric}
For every $2$-dimensional area-minimizing cone $S$ in $\mathbb R^{2+n}$ as in Definition \ref{Def:EpiperimetricInequality}, there are positive constants $\varepsilon (S), \delta (S)$ such that $S$ satisfies the $(\varepsilon,\delta)$-epiperimetric property. 
\end{Proposition}

Combined with Corollary \ref{c:compactness} we then easily conclude that

\begin{Corollary}\label{c:epi-uniform}
For every choice of integers $Q\geq 0, \bar Q\geq 1$, and $n\geq 1$ there are positive constants $\varepsilon (Q, \bar Q, n), \delta (Q, \bar Q, n)$ such that every $S\in \mathscr{C} (Q, \bar Q, n)$ satisfies the $(\varepsilon,\delta)$-epiperimetric property.
\end{Corollary}

The proof that Proposition \ref{p:epiperimetric} implies Theorem \ref{t:main-preciso-2} given in \cite{HM} applies line by line to our case, with the caveat that, in order to achieve the independence of all the constants from the particular cone (and replace it with the density of the cone at the origin and the number $Q$) it suffices to apply Corollary \ref{c:epi-uniform}

White in \cite{White} was the first to prove the epiperimetric property for every $2$-dimensional area-minimizing cone without boundary, which we can interpret as the case $Q=0$ of Proposition \ref{p:epiperimetric}, in particular from now on we adopt the latter convention. Hirsch and Marini in \cite{HM} proved the case $Q=1$. From both statements we gather than the validity of Corollary \ref{c:epi-uniform} when $Q\in \{0,1\}$. The general case can then be concluded from the following decomposition lemma.

\begin{Lemma}\label{l:decomposition}
Let $\varepsilon_0>0$ be any given positive number and $\bar Q \geq Q \geq 2, n\geq 1$ be any given triple of integers. Then there is $\varepsilon>0$ with the following property. Assume that $S\in \mathscr{C} (Q, \bar Q, n)$ and that $Z$ is a $1$-dimensional integral current in $\partial \bB_1$ which satisfies the assumptions of Definition \ref{Def:EpiperimetricInequality}. Then there are 
\begin{itemize}
\item[(A)] $Q$ cones $S_i = R_i \ttimes \a{0} \in \mathscr{C} (1, \bar Q, n)$,
\item[(B)] $Q+1$ integral currents $Z_i$ in $\partial \bB_1$, 
\end{itemize}
such that
\begin{itemize}
\item[(i)] $Z = \sum_i Z_i$;
\item[(ii)] $\|S\| (\bB_1) = \sum_i \|S_i\| (\bB_1)$, $\|Z\| = \sum_i \|Z_i\|$;
\item[(iii)] $\partial Z_i = \partial R_i$ for all $i\in \{1, \ldots, Q\}$;
\item[(iv)] $\mathcal{F} (Z_i - R_i) \leq \varepsilon_0$ for all $i\in \{1, \ldots, Q\}$;
\item[(v)] $\|Z_i\| (\partial \bB_1) - \|R_i\| (\partial \bB_1) \leq \varepsilon_0$ for all $i \in \{1, \ldots, Q\}$;
\item[(vi)] $\operatorname{dist}(\operatorname{spt}(Z_i), \operatorname{spt}(R_i)) \leq \eps_0$ for all $i\in \{1, \ldots, Q\}$;
\item[(vii)] $\partial Z_{Q+1}=0$ and $\|Z_{Q+1}\| (\partial \bB_1) \leq \varepsilon_0$.
\end{itemize}
\end{Lemma} 

To pass from Lemma \ref{l:decomposition} to Proposition \ref{p:epiperimetric} we simply use the case $Q=1$ of Corollary \ref{c:epi-uniform} to each of the pieces $Z_i$ with $i\leq Q$ to find a suitable competitor $H_i$ with $\partial H_i = \partial Z_i + \a{\ell}\res \bB_1$, while we can use the isoperimetric inequality to find $H_{Q+1}$ such that $\partial H_{Q+1} = Z_{Q+1}$ and $\mathbf{M} (H) \leq C (\mathbf{M} (Z_{Q+1}))^2$. The verification that $H = \sum_i H_i$ gives the epiperimetric inequality for a suitable choice of the parameters involved is elementary and left to the reader. 

\begin{proof}
We argue by contradiction and assume that $\{Z^j\}$ is a sequence of $1$-dimensional integral currents in $\partial \bB_1$ and $S^j$ is a sequence of cones in $\mathscr{C} (Q, \bar Q, n)$ such that the assumptions in Definition \ref{Def:EpiperimetricInequality} hold true with a sequence of vanishing $\varepsilon_j$, but the decomposition claimed in the Lemma does not exist for every $j$. Since $\mathscr{C} (Q, \bar Q, n)$ is a compact set, we can assume, by extraction of a subsequence, that $S^j$ converges to some $S^\infty$ in $\mathscr{C} (Q, \bar Q, n)$. Observe that the only condition relating the pieces $S^j_i$ to $S^j$ is that $\|S^j\| (\bB_1) = \sum_i \|S^j_i\| (\bB_1)$. 
However, since $\|S^\infty\| (\bB_1), \|S\| (\bB_1) \in \{ \frac{k}{2} \pi: k \in \Z, k \leq \bar Q \}$ (this is true for all the cones in $\mathscr{C} (Q, \bar Q, n)$) and since $\|S^j\| (\bB_1) \to \|S^\infty\| (\bB_1)$, we can assume, without loss of generality, that the $\|S^j\| (\bB_1) = \|S\| (\bB_1)$. In particular we can substitute  $S^j$ with $S$ in the contradiction statement. Denoting by $R$ the cross section of $S$, we have therefore a sequence of currents which satisfies the following conditions
\begin{itemize}
\item[(a)] $\partial Z^j = \partial R$;
\item[(b)] $Z^j$ converges to $R$ in the flat norm;
\item[(c)] ${\rm spt}\, (Z^j)$ converges to ${\rm spt}\, (R)$ in the Hausdorff distance;
\item[(d)] $\limsup_j \|Z^j\| (\partial \bB_1) \leq \|R\| (\partial \bB_1)$.
\end{itemize}
The contradiction assumption is that for some $\varepsilon_0$, independent of $j$, there is no decomposition as claimed in the lemma. However, will construct one for $j$ large enough reaching the desired contradiction.

Observe that by (b) and the semicontinuity of the mass, (d) can actually be improved to
\begin{equation}\label{e:convergence_mass}
\lim_j \|Z^j\| (\partial \bB_1) = \|R\| (\partial \bB_1)\, .
\end{equation}
Let $P$ and $N$ be the two points in $\partial \bB_1$ such that $\partial R = Q \a{N} - Q \a{P}$.
Recall the decomposition theorem in \cite[4.2.25]{Federer}: each $Z^j$ can be decomposed in an (at most) countable sum of indecomposable $1$-dimensional integral currents $\sum_i T^j_i$ with the following properties 
\begin{align*}
Z^j &= \sum_i T^j_i\,, \\
\|Z^j\| &= \sum_i \|T^j_i\| \,, \\
\|\partial Z^j\| &= \sum_i \|\partial T^j_i\|\, .
\end{align*}
It turns out that, for each $T^j_i$, either $\partial T^j_i = 0$ or there is a positive integer $k\leq Q$ such that $\partial T^j_i = k \a{N}- k \a{P}$. \cite[4.2.5]{Federer} also implies that each $T^j_i$ is given by $(f^j_i)_\sharp \a{[0, \mathbf{M} (T^j_i)]}$ for some Lipschitz map $f_i^j: \mathbb R\to \partial \bB_1$, which is injective on $[0, \mathbf{M} (T^j_i))$. In particular, if $\partial T^j_i \neq 0$, then we must have $\partial T^j_i = \a{N} - \a{P}$. Upon reordering the $T^j_i$ we can thus assume $\partial T^j_i = \a{N} - \a{P}$ for all $1\leq i \leq Q$ and $\partial T^j_i= 0$ for all $i\geq Q+1$. We then set $T^j_\sharp := \sum_{i\geq Q+1} T^j_i$. 

So far we have
\[ Z^j = T_\sharp^j + T_1^j + \sum_{i=2}^Q T_i^j \]
with 
\[ \| Z^j \| = \| T_\sharp^j\| + \| T_1^j \| + \sum_{i=2}^Q \|T_i^j\| \,. \]
By the compactness of integral currents and \eqref{e:convergence_mass}, we can assume that a subsequence, not relabeled, satisfies $T^j_i \to T_i$ and $T^j_\sharp \to T_\sharp$ in the sense of currents, for some integral currents $T_i$ and $T_\sharp$. We claim that we can decompose $T^j_\sharp$ into $T^j_\flat + T^j_\natural$ with
\begin{itemize}
	\item $\partial T^j_\flat = \partial T^j_\natural = 0$; 
	\item the Hausdorff limit of ${\rm spt}\, (T^j_\flat)$ (which upon extraction of a subsequence we can assume to exist) is contained in ${\rm spt}\, (T_\sharp)$;
	\item $\mathbf{M} ( T^j_\natural) \to 0$. 
\end{itemize}
We will then set $Z^j_1 := T^j_1 + T^j_\flat$, $Z^j_{Q+1} := T^j_\natural$, and $Z_i^j := T_i^j$ for the remaining $j \in \{2, \dots, Q\}$. Clearly $(i)$, $(ii)$, $(iii)$, and $(vii)$ hold, it remains to show $(iv)$-$(vi)$.

We first show the decomposition of $T_\sharp^j$. Observe that we have two possibilities
\begin{itemize}
	\item[(I)] ${\rm spt}\, (T_\sharp)$ and ${\rm spt}\, (R-T_\sharp)$ are disjoint;
	\item[(II)] ${\rm spt}\, (T_\sharp)$ and ${\rm spt}\, (R-T_\sharp)$ are not disjoint. 
\end{itemize}
In the first case choose then $\eta>0$ so that the $\eta$-neighborhoods of the two pieces are disjoint. Since the supports of the $Z^j$ are converging to the support of $R$ in the sense of Hausdorff, it follows that, if $j$ is large enough, then the support of each current $T^j_i$ (which is a connected set) is either contained in the $\eta$-neighborhood of ${\rm spt}\, (T_\sharp)$ or it is contained in the $\eta$-neighborhood of  ${\rm spt}\, (R-T_\sharp)$. Then $T^j_\flat$ is defined as the sum of those $T^j_i$ with $i\geq Q+1$ whose supports are contained in the $\eta$-neighborhood of ${\rm spt}\, (T_\sharp)$.

In case (II), however, Proposition \ref{p:classification} implies that the support of the whole current $R$ is a single great circle $\gamma$, which equals the support of $T_\sharp$. But then the Hausdorff limit of ${\rm spt}\, (T^j_\sharp)$ is necessarily contained in $\gamma$.

Now we turn to $(iv)$-$(vi)$. Note that the compactness gives
\begin{align}\label{e:flat_convergence}
\mathcal{F} (T^j_i - T_i) + \mathcal{F} (T^j_\sharp - T_\sharp) \to 0
\end{align}
and thus 
\[
R =  T_\sharp + \sum_{i=1}^Q T_i \, .
\]
It follows in particular that $\partial T_i = \a{N} - \a{P}$, while $\partial T_\sharp = 0$, but also that
\[
\mathbf{M} (R) \leq \mathbf{M} (T_\sharp) +\sum_i \mathbf{M} (T_i) \, .
\]
Next note that, by semicontinuity of the mass and \eqref{e:convergence_mass}, we must have
\[
\mathbf{M} (T_\sharp) + \sum_i \mathbf{M} (T_i) \leq \liminf_{j\to \infty} \mathbf{M} (Z^j)\leq \mathbf{M} (R)\, .
\]
It thus follows that 
\begin{align}
\mathbf{M} (R) & = \mathbf{M} (T_\sharp) + \sum_i \mathbf{M} (T_i) \,, \label{e:mass_R_splits}\\
\mathbf{M} (T_i) &= \lim_j \mathbf{M} (T^j_i) \,, \label{e:mass_convergence_Ti} \\
\mathbf{M} (T_\sharp) &= \lim_j \mathbf{M} (T^j_\sharp)\, .
\end{align}
Notice that
\begin{align*}
	S= \a{0} \cone R = \underbrace{\a{0} \cone T_\sharp + \a{0} \cone T_1}_{=: S_1} + \sum_{i=2}^Q \underbrace{\a{0} \cone T_i}_{=:S_i} \,.
\end{align*}
Then we have by \eqref{e:mass_R_splits}
\[ \|S\| =  \sum_{i=1}^Q \|S_i \| \]
and as $S$ is area-minimizing, so are the $S_i$'s.
Each $S_i$ is then an element of $\mathscr{C} (1, \bar Q, n)$. 

Fix now an $i\geq 2$ and recall \eqref{e:flat_convergence} and \eqref{e:mass_convergence_Ti}. We claim $(iv)$-$(vii)$ hold true for $Z_i^j$ (with $j$ large enough) and $R_i := T_i$. Indeed, recall that each $T^j_i = (f^j_i)_\sharp \a{[0, \mathbf{M} (T^j_i)]}$ for a Lipschitz injective curve $f^j_i :[0, \mathbf{M} (T^j_i)] \to \partial \bB_1$. Assuming that $f^j_i$ is parametrized with constant speed over the interval $[0,1]$, we have $\sup_j \|\dot{f}^j_i\|_\infty < \infty$ and $f^j_i (0) = P$, $f^j_i (1) = N$. We can apply Ascoli-Arzel\'a and assume, without loss of generality, that $f^j_i$ converges to a Lipschitz  map $f_i : [0,1]\to \partial \bB_1$ and it thus follows that $R_i = T_i = (f_i)_\sharp \a{[0,1]}$. Observe that 
\[
\mathbf{M} (R_i) \leq \int_0^1 |\dot{f}_i| \leq \liminf_{j\to \infty} \int_0^1 |\dot{f}^j_i| \leq \liminf_j \mathbf{M} (T^j_i)\, .
\] 
Therefore, the chains of inequalities above are actually equalities, from which we infer that ${\rm spt}\,(R_i) = f_i ([0,1])$. In particular we have that 
\begin{align*}
{\rm dist}_{\cH} ({\rm spt}\, (T^j_i), {\rm spt}\, (R_i))
= {\rm dist}_{\cH} (f_i^j ([0,1])), f_i ([0,1]))
\to 0 \qquad \text{ for all } i\geq 2.
\end{align*}
This shows $(vi)$ for all $i\geq 2$ (recall that we set $Z^j_i := T^j_i$). 

The only choice for $R_1$ is then
\[ R_1 = R - \sum_{i=2}^Q R_i = T_\sharp + T_1 \,. \]
Observe that while we have 
\begin{align*}
\mathcal{F} (T^j_1 + T^j_\sharp - R_1) \to 0 
\qquad \text{ and } \qquad
\mathbf{M} (T^j_1 + T^j_\sharp) \rightarrow \mathbf{M} (R_1)\,,
\end{align*}
we cannot conclude that the supports of $T^j_1 + T^j_\sharp$ converge in the sense of Hausdorff to the support of $R_1$, because we do not know the representation by the Lipschitz map for $T_\sharp^j$. However, we can infer that 
\begin{align*}
{\rm dist}_{\cH} ({\rm spt}\, (T^j_1), {\rm spt}\, (T_1)) \to 0\,. 
\end{align*}
Thus, we achieve our goal, because we know that $T_\natural^j$ converges to zero and $T^j_\flat$ converges to $T_\sharp$, which in turn implies that ${\rm spt}\, (T_\sharp)$ must be necessarily contained in the Hausdorff limit of ${\rm spt}\, (T^j_\flat)$.
 
\end{proof}

\bibliographystyle{abbrv}
\bibliography{UniqTgConesBoundaryCSS1.bib}

\end{document}